\documentclass[11pt, twoside]{amsart}
\usepackage{epsfig, amssymb,  amsfonts,xspace,amsthm,amscd}
\usepackage{mathtools}
\usepackage[latin1]{inputenc}
\usepackage[T1]{fontenc}
\usepackage[all]{xy}
\usepackage{hyperref}
\numberwithin{equation}{section}

\newtheorem{lemma}{Lemma}[section]

\newtheorem{theorem}[lemma]{Theorem}

\newtheorem{rem}[lemma]{Remark}

\newtheorem*{special theorem}{My Specially-Named Theorem}

\newcommand{\comment}[1]{}

\begin{document}


\pagestyle{plain}


\title{The quotients of the $p$-adic group ring of a cyclic  group of  order $p$}

\date{\today}

\author[b]{Maria Guedri$^{*}$}
\author[c]{Yassine Guerboussa$^{\dag}$}

\address[$^{*}$]{Department of Mathematics, Fr\`{e}res Mentouri University-Constantine 1, Ain El Bey road, 25017 Constantine. Algeria
 \\ {\tt Email: guedri\_maria@yahoo.fr}}
\address[$^{\dag}$]{Department of Mathematics, University Kasdi Merbah Ouargla, Ouargla, Algeria \\ {\tt Email: yassine\_guer@hotmail.fr}}
\maketitle

\begin{abstract}
We classify, up to isomorphism, the  $\mathbb{Z}_pG$-modules of rank $1$ (i.e., the quotients of $\mathbb{Z}_pG$) for $G$  cyclic of order $p$, where $\mathbb{Z}_p$ is the ring of $p$-adic integers.  This allows us in particular to determine effectively the quotients of $\mathbb{Z}_pG$ which are cohomologically  trivial over $G$. There are natural zeta functions associated  to $\mathbb{Z}_pG$ for which we give an explicit formula. 

\end{abstract}\vspace{1cm}
	
	
\normalsize

 \section{Introduction}
 The results in this note have arisen from  the investigation of \emph{cohomologically trivial} (\emph{CT} for short) modules over finite $p$-groups.   The interest in the latter has been motivated in recent years  by a conjecture of Schmid \cite[p.3]{Schmid} stating that:
\textit{for every finite non-abelian $p$-group $P$, the center of the Frattini subgroup $Z(\Phi(P))$  is not CT over $P/\Phi(P)$}.   Although the latter holds  for several families 
 of $p$-groups (see \cite{Schmid, AliPowerfulCT, Benmoussa, AMY}),  it fails to hold in general as is shown in Abdollahi \cite{AliCT}; more precisely, counter-examples are provided by the $2$-groups of order $2^8$ with {\sf IdSmallGroup} $298,\dots,307$  in  GAP \cite{GAP}.  The interesting issue here is that such counter-examples are rare; e.g., none of them occur among the $2$-groups of order $2^9$, noting that there are $10494213$ groups of the latter order facing $56092$ ones of order $2^8$.  Hence an in-depth analysis is needed  for a better understanding of such counter-examples,  or, more generally, of the finite CT modules (over finite $p$-groups).

For a finite  $p$-group $G$, every  $G$-module $A$ that is finite of $p$-power order  can be viewed in a canonical way as  a $\mathbb{Z}_pG$-module  (where $\mathbb{Z}_p$ is the ring of $p$-adic integers).  In order to deal with presentations of such $A$'s  we may work more generally within the category of finitely generated   $\mathbb{Z}_pG$-modules.  Note that $\mathbb{Z}_pG$   is a local ring whose maximal ideal is given by 
$\mathfrak{m}=\mathfrak{I}_G+p\,\cdot 1$,  where $\mathfrak{I}_G$ is the augmentation ideal of $\mathbb{Z}_pG$.  This implies in particular, by using Nakayama's lemma, that, for every finitely generated $\mathbb{Z}_pG$-module $A$, all the minimal generating sets of $A$  have a common size equaling
\begin{equation}\label{Definition of the rank of A}
r(A):=\dim_{\mathbb{F}_p} A/\mathfrak{m}A=\dim_{\mathbb{F}_p} A/[A,G]+pA.
\end{equation}
We may call $r(A)$ the \textit{$\mathbb{Z}_pG$-rank} of $A$ (or simply the \textit{rank} of $A$ if no confusion arises).   If we write $d(A)$ for
the minimal number of generators of $A$ over $\mathbb{Z}_p$, that is, $d(A)=\dim_{\mathbb{F}_p}A/pA$, we  obtain immediately 
\begin{equation}\label{Two formulas for the rank of A}
r(A)=d(A/[A,G]).
\end{equation}
(For a detailed discussion of such invariants see \cite[Sec. 2]{MY}.)

This note treats the simplest possible case, i.e., that where $G=\langle g \rangle$ is cyclic of order $p$ and $A$ is a $\mathbb{Z}_pG$-module of rank $1$.  Note that the cohomology groups $\hat{H}^*(G,A)$ can be treated in a fairly elementary way since
\begin{equation}\label{Cohomology groups}
\hat{H}^{2n}(G,A)=A^G/NA \quad \mbox{and } \quad \hat{H}^{2n+1}(G,A)=\, _{N}A/[A,G], 
\end{equation}
where $N=1+\cdots+g^{p-1}$ is the norm defined by $G$, and $_{N}A$ is the kernel of multiplication $N$ (cf. e.g., Serre \cite[Chap VIII, \S 4]{SerreLF}).  In this case, saying that $A$ is CT amounts to saying that $\hat{H}^n(G,A)=0$ for $n=0,1$.  Note that if $A$ is finite, the latter two groups are isomorphic (i.e., the Herbrand quotient equals $1$), and so for $A$ to be CT it suffices that $\hat{H}^n(G,A)=0$ for just one $n$. 

 Our primary intention was to get a classification (up to isomorphism) of all those $A$ which are CT, but it turned out that treating all the cases simultaneously requires  no  extra effort. 
 
  In the case where $A$ is infinite, we may consider the torsion part $$T(A)=\{a\in A\mid p^la=0 \mbox{ for some integer } l>0\}.$$
  Accordingly, we have the subdivision  $T(A)=0$ (i.e., $A$ is torsion-free over $\mathbb{Z}_p$) and $T(A)\neq 0$.  In the first subcase, we shall see that  $A\cong \mathbb{Z}_pG$,  $A\cong \mathfrak{I}_G$ (the augmentation ideal), or $A\cong \mathbb{Z}_p$ (a trivial $G$-module). In the second subcase, we have either $A\cong L/p^nL^G$ or 
  $A\cong L/(g-1)^n\mathfrak{I}_G$, where $L=\mathbb{Z}_pG$ and $p^n$ is the order of $T(A)$ (see Theorem \ref{Classification of the in finite case}).  Note that $A$ is CT here if and only if $A\cong \mathbb{Z}_pG$.

  The case where $A$ is finite---which is more interesting for us---needs more elaboration.  Roughly speaking, every such a   $A$ is  determined by three parameters $r$, $s$ and $t$, where 
$$p^r=\mbox{ order of } [A,G], \qquad p^s= \mbox{ order of } A^G,$$
and $t$ is an element of the finite field $ \mathbb{F}_p$  which satisfies a relation of the form
$$
p^sa=\varepsilon^s((g-1)^{s(p-1)}+t(g-1)^{r}) b,
$$
with $a$ denoting  a generator of $A$ over  $\mathbb{Z}_pG$, $b=(g-1)a$, and  $\varepsilon$ is a certain  operator on the augmentation ideal $\mathfrak{I}_G$.  The precise statement is given in Theorem \ref{Classification of cyclic modules}.  The fact that the parameter $t\in \mathbb{F}_p$ is well defined (i.e., independent from the choice of the generator $a$) will be clear after Lemma \ref{Unicity of the invariant t}.  It is worth noting that the order of $A$ is determined by $r$ and $s$ (more precisely, $\vert A\vert=p^{r+s}$), and that the cohomological behavior   of $A$ is determined by $t$.

Writing $\hat{b}_n$ for the number of the $\mathbb{Z}_pG$-modules of rank $1$ and order  $p^n$,  and $\hat{c}_n$ for the  number of the latter which are in addition  CT, it follows by Theorem \ref{Classification of cyclic modules} that  (cf. Formulae (\ref{Number of finite quotient and finite CT quotients over G cyclic of order p}) and (\ref{Number of finite CT quotients over G cyclic of order p}))
$$\hat{b}_n =1+p(n-1) \quad \mbox{and} \quad \hat{c}_n =(p-1)(n-1),$$
for all $n\geq 2$.  Subsequently the density of CT modules within $\mathbb{Z}_pG$ is given  by
\begin{equation}\label{Density of CT modules}
\lim_n \dfrac{\hat{c}_n}{\hat{b}_n}  =1-p^{-1}.
\end{equation}
It seems likely that similar formulae hold for the finite quotients of $L=(\mathbb{Z}_pG)^d$, for every  integer $d>0$ and every finite $p$-group $G$.
 This question is closely related to studying the zeta functions:
\begin{equation}\label{Defn of zeta functions}
Z(L;s)=\sum_{n\geq 1}b_n(L)p^{-ns}   \quad \mbox{and }\quad Z_{\text{coh}}(L;s)=\sum_{n\geq 1}c_n(L)p^{-ns},
\end{equation}
where $b_n(L)$  denotes the number of  all $\mathbb{Z}_pG$-submodules of $L$ of index $p^n$, and $c_n(L)$   the number of the latter which are in addition free.  It is worth mentioning   that a $\mathbb{Z}_pG$-submodule $M\subseteq L$ is free if and only if  $L/M$ is CT over $G$; this follows, e.g., from \cite[Chap. IX, Thm. 7]{SerreLF} together with the fact that $\mathbb{Z}_pG$ is local; see  \cite[Sec. 3]{MY} for  more details.

In the case $L=\mathbb{Z}_pG$ with $G$ having order $p$,  Theorem \ref{Classification of cyclic modules} (i) implies  that $b_n(L) =\hat{b}_n$ and $c_n(L)=c_n$ for all $n$, from which it is readily seen that
\begin{equation}\label{Explicit Zetafunction 1}
Z(L;s)=\dfrac{1-p^{-s}+p^{1-2s}}{(1-p^{-s})^2},
\end{equation}
and
\begin{equation}\label{Explicit Zetafunction 2}
Z_{coh}(L;s)= \dfrac{(1-p)p^{-2s}}{(1-p^{-s})^2}.
\end{equation}
Note that (\ref{Explicit Zetafunction 1}) has been already established by other methods in Solomon \cite{Solomon},  Reiner \cite{Reiner1980}, and Bushnell-Reiner \cite{Bushnell and Reiner}.

\begin{rem}
\textnormal{ The zeta functions (\ref{Defn of zeta functions}) are analogous to the ones studied in the context of subgroup growth; see Lubotzky and Segal \cite{LuboSegal}. In fact, our interest in (\ref{Defn of zeta functions})   is largely influenced  by the philosophy  presented in the latter monograph.  	The natural question that arises when considering such  series is  whether they are given by rational functions in $p^{-s}$.  The second author has already posted this question on MathOverflow \cite{MO}; shortly after, we realized that the zeta functions (\ref{Defn of zeta functions})   had been  already introduced in Solomon \cite{Solomon}, who proved in particular their rationality (using combinatorial methods).  It is worth mentioning that Solomon's work predates the appearance of the seminal paper \cite{GrunSegalSmith} in which 
	the subgroup zeta functions were first considered.  Solomon's work was followed up in Bushnell and Reiner \cite{Bushnell and Reiner}, who gave another proof of the rationality following Tate's thesis; for another approach  (using Iguza's zeta functions) we refer to  \cite{MY}.}
\end{rem}

About the notation, notice that $A^G$ denotes here  the submodule of the $G$-fixed elements in $A$ (unlike   \cite{AliCT, AMY, Schmid} where this was denoted by $A_G$).  The remaining notation  need no explanation.   In Sec. 2, we treat the finite case, and prove in particular Theorem \ref{Classification of cyclic modules}.  The last section treats the infinite case, i.e., Theorem \ref{Classification of the in finite case} and its proof.

The arguments presented in this note might be extended to more general $p$-adic rings (instead of $\mathbb{Z}_p$), i.e., valuation rings of finite extensions of the field $\mathbb{Q}_p$.

\section{The finite case}\label{Cyclic module over cyclic groups of order p}
Throughout this section we suppose that $G=\langle g\rangle$ is cyclic of order $p$ and that  $A$ is a finite $\mathbb{Z}_pG$-module of rank $1$, that is, $r(A)=1$.   To ease the notation,  we write   $B$ for $[A,G]$, and    $B_j$ for $[B,G,\ldots,G]$, where $G$ appears $j$ times.   It is readily seen that  $B_j=(g-1)^jB$, for all $j\geq 0$ (cf. \cite[Sec. 3]{BG}). Set
$$p^r=\vert B\vert \quad \mbox{ and }\quad  p^s=\vert A^G\vert.$$
Observe that multiplication by $g-1$ induces a $G$-isomorphism $A/A^G\cong B$, and so $\vert A\vert =p^{r+s}$.

The fact that $B$ is annihilated by the norm $N=1+\cdots+g^{p-1}$ implies that $B$ has a natural structure of a module over $\mathfrak{O}=\mathbb{Z}_p[\zeta]$, where $\zeta$ is a primitive $p$th root of $1$,  and $\zeta$ acts on $B$ as $g$ (cf. \cite[Remark 1.1]{BG}).  Hence, for all $j\geq 1$, we have
\begin{equation}\label{Central series of B as an O-module}
B_j=(\zeta-1)^jB=\mathfrak{P}^jB,
\end{equation}
where $\mathfrak{P}=(\zeta-1)\mathfrak{O}$ is the unique maximal ideal of $\mathfrak{O}$.  

Let $a$ be a generator of $A$, and let $b=(g-1)a$.  It follows by (\ref{Two formulas for the rank of A}) that $A/B$ is a cyclic $p$-group  generated by the canonical image of $a$, and so $A=\langle a\rangle+B$.  Multiplying the latter by $g-1$ yields $B=\langle b\rangle+B_1$, which implies that $B$ is a cyclic $\mathfrak{O}$-module.  Next, as $\mathfrak{O}$ is a discrete valuation ring, we obtain
\begin{equation}\label{B as an O-module}
B\cong \mathfrak{O}/\mathfrak{P}^{r} \quad \mbox{and } \quad B_j\cong
\mathfrak{P}^{\min\{j,r\}}/\mathfrak{P}^{r}.
\end{equation}
In particular, $r$ is the smallest integer such that $(g-1)^rB=0$.  Also, as $A/B$ has order $p^s$ (because $\vert A\vert=\vert B\vert\vert A^G\vert$), we have \begin{equation}\label{Structure of the quotient A/B}
A/B\cong \langle a\rangle/\langle a\rangle\cap B\cong \mathbb{Z}/(p^s).
\end{equation}
It follows that $p^s$ is the smallest integer such that $p^sa\in B$ (that is, $p^sa$ is a generator of the cyclic $p$-group
$\langle a\rangle\cap B$).  

Observe that the case $s=0$ occurs exactly when  $A^G=0$, or equivalently $A=0$. The case  $r=0$ corresponds to the fact that $A$ is a cyclic $p$-group of order $p^s$ on which $G$ acts trivially.

We may suppose in the sequel that $s,r\geq 1$.  Using  (\ref{B as an O-module}), we see that there exists a polynomial $f(\zeta)\in \mathfrak{O}$,  unique modulo $\mathfrak{P}^r$, such that
\begin{equation}\label{The polynomial characterizing p^sa}
p^sa=f(\zeta) b.
\end{equation}
Before giving an explicit formula for  $f(\zeta)$, recall  
 the well-known fact that there is a unit $\varepsilon\in \mathfrak{O}^{\times}$  such that $p=\varepsilon(\zeta-1)^{p-1}$. The latter implies in particular  that
\begin{equation}\label{Power and commutator in B}
p^jc=\varepsilon^j (\zeta-1)^{j(p-1)}c
\end{equation}
for all  integers $j\geq 1$ and all $c\in B$.
\begin{lemma}\label{Formula for the polynomial characterizing p^sa}
	Suppose that $r,s\geq 1$.  Then
	there exists a unique element $t\in \mathbb{F}_p$  such that
	$$f(\zeta)=\varepsilon^s(\zeta-1)^{s(p-1)-1}+t(\zeta-1)^{r-1} \mod \mathfrak{P}^r.$$
	
\end{lemma}
\begin{proof}
	Multiplying the equation (\ref{The polynomial characterizing p^sa}) by $g-1$ yields
	$$p^sb=(\zeta-1)f(\zeta)b,$$
	from which we get by   (\ref{Power and commutator in B})
	$$\left[ \varepsilon^s(\zeta-1)^{s(p-1)}-(\zeta-1)f(\zeta)\right ]b=0.$$
Using (\ref{B as an O-module}), we infer
\begin{equation}\label{Aux Eq}
(\zeta-1)\left [f(\zeta)-\varepsilon^s(\zeta-1)^{s(p-1)-1}\right]=t(\zeta-1)^{r}+o(\zeta)(\zeta-1)^{r+1}, 
\end{equation}
where $t\in \mathbb{F}_p$ and $o(\zeta) \in \mathfrak{O}$. The fact that $t$ can be chosen in $\mathbb{F}_p$ follows because  $\mathfrak{P}^{r}/\mathfrak{P}^{r+1}\cong\mathbb{F}_p$.  Such a $t$ is uniquely determined as if we have another $t'\in \mathbb{F}_p$ and another $o'(\zeta) \in \mathfrak{O}$ such that
$$t(\zeta-1)^{r}+o(\zeta)(\zeta-1)^{r+1}= t'(\zeta-1)^{r}+o'(\zeta)(\zeta-1)^{r+1},$$
then  
$$(t'-t)
(\zeta-1)^{r}=0 \mod \mathfrak{P}^{r+1},$$
and so $t'=t \mod p$.

Next, as $\mathfrak{O}$ is a domain,  we can simplify $(\zeta-1)$ from both sides of (\ref{Aux Eq}), so that
	$$f(\zeta)=\varepsilon^s(\zeta-1)^{s(p-1)-1}+t(\zeta-1)^{r-1}+o(\zeta) (\zeta-1)^{r},$$
which concludes the proof.
\end{proof}
The next lemma shows that the above polynomial $f(\zeta)$, or equivalently  $t\in \mathbb{F}_p$, is independent from the choice of the generator  $a$ of $A$.

\begin{lemma}\label{Unicity of the invariant t} Suppose that $r,s\geq 1$.   If  $a'$ is a generator  of $A$ and  $b'=(g-1)a'$, then $$p^sa'=f(\zeta)b'.$$
\end{lemma}
\begin{proof}
	Set $h(\zeta)=\varepsilon^s(\zeta-1)^{s(p-1)-1}$. Arguing as in the previous lemma, we see that $p^sa'=f'(\zeta)b'$, where  $f'(\zeta)\in \mathfrak{O}$ has the form 
$$
f'(\zeta)=h(\zeta)+t'(\zeta-1)^{r-1} \mod \mathfrak{P}^r,
$$
for some $t'\in  \mathbb{F}_p$.   As $a'+B$ is a generator of the cyclic $p$-group $A/B$,  there exist integers 	$k$ and $l$,  with $0< k< p-1$, and $c\in B$ such that
\begin{equation}\label{Aux (*)}
a'=(k +l p)a+c. 
\end{equation}
Multiplying the latter by $p^s$ and using the previous lemma, we obtain
\begin{align*}
p^sa'&=(k +l p)p^sa+p^sc \\
&=(k+lp)h(\zeta)b+(k+lp)t(\zeta-1)^{r-1}b+ p^sc.
\end{align*}
Clearly, $p(\zeta-1)^{r-1}b=0$, so by (\ref{Power and commutator in B}) we have $p^sc=h(\zeta)(\zeta-1)c$. Therefore,
$$p^sa'=(k+lp)h(\zeta)b+kt(\zeta-1)^{r-1}b+ h(\zeta)(\zeta-1)c. $$
On the other hand, multiplying (\ref{Aux (*)})  by $(g-1)$ shows that
$$
b'=(k +l p)b+(\zeta-1)c. 
$$
Substituting the latter in   $p^sa'=f'(\zeta)b'$, knowing that $p(\zeta-1)^{r-1}b=0$ and $t'(\zeta-1)^{r}c=0$, we obtain
$$
p^sa'=(k+lp)h(\zeta)b+kt'(\zeta-1)^{r-1}b+ h(\zeta)(\zeta-1)c.$$
By comparing the above two expressions of $p^sa'$, we infer
$$k(t-t')(\zeta-1)^{r-1}b=0.$$
As $0<k<p$, the factor $t-t'$ should be divisible by $p$, that is, $t=t'$.  Thus $f'(\zeta)=f(\zeta)$, as desired.
\end{proof}

The above discussion shows that we have an invariant $t\in \mathbb{F}_p$ of $A$ defined whenever $r,s\geq 1$.  It is convenient to extend the definition of $t$ to include the case $rs=0$ by setting  $t=0$.

Before embarking on the main theorem, let us prove the following, which will be  useful in determining the cohomological type of $A$.  Recall that  $N=1+g+\ldots+g^{p-1}$.

\begin{lemma}\label{Cohomology type of A in terms of the norm}
Let $a$ be a generator of $A$, and let $b=(g-1)a$. Then for every integer $j\geq 1$ we have
	$$ Np^ja=p^{j+1}a-\varepsilon^{j+1} (\zeta-1)^{(j+1)(p-1)-1}b.$$
\end{lemma}

\begin{proof}
It is readily seen that 
$N=\sum_{i=1}^{p} \binom{p}{i} (g-1)^{i-1}$; thereby
\begin{eqnarray*}
Np^ja&=&\sum_{i=1}^{p} \binom{p}{i}p^j (g-1)^{i-1}a\\
&=& p^{j+1}a+\sum_{i=2}^{p} \binom{p}{i} (g-1)^{i-2}p^jb\\
&=& p^{j+1}a+\sum_{i=2}^{p} \binom{p}{i} \varepsilon^j (g-1)^{j(p-1)+i-2}b \qquad \mbox{(by   (\ref{Power and commutator in B}))}\\
	&=& p^{j+1}a+\varepsilon^j (g-1)^{j(p-1)-1}  \sum_{i=2}^{p} \binom{p}{i} (g-1)^{i-1}b\\
&=& p^{j+1}a+\varepsilon^j (g-1)^{j(p-1)-1}  (N-p)b.
\end{eqnarray*}
To conclude the proof,  observe that
 $Nb=0$ and $pb=\varepsilon(g-1)^{p-1}b$  (observe also that $g$ and $\zeta$ are interchangeable since  they act in the same way on $B$).
\end{proof}

Note that the identity in the latter lemma can be written more concisely as
\begin{equation}
Np^ja=p^{j+1}a-\varepsilon p^j(\zeta-1)^{p-2} b, \qquad \mbox{ for } j\geq 1.
\end{equation}
Before stating the main theorem, note
 that the augmentation ideal $\mathfrak{I}_G$  is an $\mathfrak{O}$-module in the obvious way (since $\mathfrak{I}_G$ is annihilated by the norm $N$); in fact, we have   $\mathfrak{I}_G\cong \mathfrak{O}$ both as $\mathfrak{O}$-modules and  as $\mathbb{Z}_pG$-modules (observe that $g$ acts on $\mathfrak{O}$ as multiplication by $\zeta$).  It follows in particular that multiplying elements of $\mathfrak{I}_G$ by the unit $\varepsilon \in\mathfrak{O}^{\times}$ makes sense.

\begin{theorem}\label{Classification of cyclic modules}
Let $G=\langle g\rangle$ be a cyclic $p$-group of order $p$, and  $A$ be a finite 	$\mathbb{Z}_pG$-module of rank $1$ with invariants $r$, $s$ and $t$.  Set $L=\mathbb{Z}_pG$. Then:
\begin{itemize}
	\item[(i)] There exists a unique submodule $M\subseteq L$ such that $A\cong L/M$;  namely
	\begin{itemize}
		\item[--] $M=L$ if $s=0$,
		\item[--]  $M=\left \langle g-1\, ,\, p^s\right \rangle$ if $r=0$ and $s\geq 1$, and 
			\item[--] 	$M=\left \langle (g-1)^{r+1}\, ,\, p^s-\varepsilon^s(g-1)^{s(p-1)}-t(g-1)^{r}\right \rangle$ if $r,s\geq 1$.
	\end{itemize}

	\item[(ii)] $A$ is completely determined by the invariants $r$, $s$ and $t$; in other words, if $A'$  is  a finite quotient of	$\mathbb{Z}_pG$ having the same invariant as $A$, then $A'\cong A$.
	\item[(iii)] For $r,s\geq 1$,  $A$ is CT if and only if $t\neq 0$.  
	
	(For $rs=0$, $G$ acts trivially on $A$, and so $A$ is CT if and only if  $s=0$, or equivalently $A=0$.)	
\end{itemize}
\end{theorem}
\begin{proof}	
(i) The case $s=0$ is trivial, so suppose that $s\geq 1$. Let $\varphi:L\to A$ be an arbitrary surjective $\mathbb{Z}_pG$-morphism, and let $a=\varphi(1)$.  

Suppose first that $r=0$.  Then $(g-1)a=0$ and $p^sa=0$, which implies that $M\subseteq \ker \varphi$.  Obviously, $\mathfrak{I}_G\subseteq M$, and since $L/\mathfrak{I}_G\cong \mathbb{Z}_p$  and $p^s\in M$ we should have $\vert L:M\vert\leq p^s$.  On the other hand, we know that $A$ has order $p^s$,  so $\vert L:\ker \varphi\vert=p^s$.  It follows that $M= \ker \varphi$.

Suppose now that $r\geq 1$. Using (\ref{B as an O-module}) and   Lemmas \ref{Formula for the polynomial characterizing p^sa} and \ref{Unicity of the invariant t}, we obtain 
$$(g-1)^{r+1}a=0\quad  \mbox{ and } \quad  p^sa=f(g)(g-1)a,$$
where $f(g)=\varepsilon^s(g-1)^{s(p-1)-1}-t(g-1)^{r-1}$.  Thus both $(g-1)^{r+1}$ and $p^s-f(g)(g-1)$ belong to $\ker \varphi$, which shows that $M\subseteq \ker \varphi$.  Denote by $M'$  the submodule of $L$ generated by  $(g-1)^{r+1}$; observe that
$M' \subseteq \mathfrak{I}_G\cap M$. As $\mathfrak{I}_G$ is a free $\mathfrak{O}$-module on the generator $(g-1)$ and $\mathfrak{P}=(g-1)\mathfrak{O}$, we have $\vert \mathfrak{I}_G:M'\vert=p^r$; thus $$\vert \mathfrak{I}_G:\mathfrak{I}_G\cap M\vert \leq p^r.$$
Also,  as $r,s\geq 1$, we have $p^s\in M+\mathfrak{I}_G$, but since $L/\mathfrak{I}_G\cong \mathbb{Z}_p$ the latter implies that $$\vert L:\mathfrak{I}_G+M\vert\leq p^s.$$
The above two inequalities show that
$$\vert L:M\vert=\vert L:\mathfrak{I}_G+M\vert \vert \mathfrak{I}_G:M\vert\leq p^{r+s};$$
but since $M\subseteq \ker \varphi$ and $L/\ker \varphi\cong A$ has order $p^{r+s}$, we obtain $\ker \varphi =M$.

To sum up,  the kernel of every $\mathbb{Z}_pG$-morphism $L\twoheadrightarrow A$ is equal to $M$, which means that  $M$ is the unique submodule of $L$ that satisfies $L/M\cong A$. This concludes the proof of (i).

(ii) If $A'$ has the same invariants as $A$, then,  by writing $A'\cong L/M'$, we see by  (i) that $M'=M$; thus $A'\cong A$.

(iii) Suppose that $r,s\geq 1$.  Clearly, $A$ is non CT if and only if $B<\ker N$ (cf. (\ref{Cohomology groups})), but since $A/B$ is cyclic, the latter amounts to saying that $p^{s-1}a\in \ker N$, that is, $Np^{s-1}a=0$.  By Lemma \ref{Cohomology type of A in terms of the norm}, we have
$$Np^{s-1}a=0 \iff p^{s}a=\varepsilon^{s} (\zeta-1)^{s(p-1)-1}b;$$
by virtue of (\ref{The polynomial characterizing p^sa}) and Lemmas \ref{Formula for the polynomial characterizing p^sa} and \ref{Unicity of the invariant t}, the latter is equivalent to  $t=0$; the result follows.
\end{proof}

The forging theorem allows us in particular to  determine explicitly 
the number of the $\mathbb{Z}_pG$-modules $A$ of rank $1$ and order  $p^n$,  as well as the number of those which are moreover  CT. If we denote these by $\hat{b}_n$  and $\hat{c}_n$, respectively, then:
\begin{equation}\label{Number of finite quotient and finite CT quotients over G cyclic of order p}
\hat{b}_n=\left\{
\begin{array}{ll}
1+p(n-1) \quad \mbox{ if } n\geq 2\\
1 \quad \mbox{ if } n=0,1
\end{array},
\right.
\end{equation}
and 
\begin{equation}\label{Number of finite CT quotients over G cyclic of order p}
\hat{c}_n=\left\{
\begin{array}{ll}
(p-1)(n-1) \quad \mbox{ if } n\geq 2\\
0 \quad  \mbox{ if } n=1\\
1 \quad   \mbox{ if } n=0
\end{array}.
\right.
\end{equation}
Indeed, the case  $n\leq 1$ is trivial, so suppose that $n\geq 2$.  We may write $n=r+s$, where $r$ and $s$ are obviously the possible invariants of  $A$.  Clearly, $s\geq 1$, and so $0\leq r\leq n-1$.  For $r=0$, $A$ is cyclic on which $G$ acts trivially, so we have only one isomorphism type for $A$;  in the remaining cases, each value of $r$  gives rise to $p$ isomorphism types of $A$ corresponding to the values $t=0,\ldots, p-1$; proving (\ref{Number of finite quotient and finite CT quotients over G cyclic of order p}).  Concerning $\hat{c}_n$, observe that the case $r=0$ cannot occur, so $1\leq r\leq n-1$; furthermore, each of the forgoing values of $r$ gives rise to $p-1$ non-isomorphic CT modules $A$ according to the values $t=1,\ldots,p-1$.

\section{The infinite case}
\begin{theorem}\label{Classification of the in finite case}
Let $G=\langle g\rangle$ be a cyclic group of order $p$, and  $A$ be an infinite $\mathbb{Z}_pG$-module of rank $1$.  Set $L=\mathbb{Z}_pG$.  Then the following hold.
\begin{itemize}
	\item[(i)] If $A$ has no $p$-torsion, then  $A\cong L$,  $A\cong \mathfrak{I}_G$ (the augmentation ideal), or $A\cong \mathbb{Z}_p$ (on which $G$ acts trivially).
		\item[(ii)] If $A$ has  non-trivial $p$-torsion, then either
		$A= L/p^nL^G$ or $A= L/(g-1)^n\mathfrak{I}_G$, where $p^n$ is the order of torsion part of $A$.
\end{itemize}
\end{theorem}
\begin{proof}
First,   observe that $d(A)\leq p$ (since $A$ is a quotient of $L$ and  $d(L)=p$).

(i) Suppose that $A$ is torsion-free.  If $d(A)=p$,  then, by writing $A$ as a quotient $L/M$ and using the fact that $\mathbb{Z}_p$ is a principal ideal domain, we see that $d(M)=0$,  so $M=0$ and $A\cong L$.

Next, we may suppose that $d(A)<p$.  If $[A,G]=0$, then by (\ref*{Two formulas for the rank of A}) we have $d(A)=r(A)=1$;  so $A\cong \mathbb{Z}_p$ and  $G$  acts trivially on it.  Suppose now that  $[A,G]\neq 0$.  As $[A,G]$ is annihilated by the norm $N=1+\cdots+g^{p-1}$, we may view it as a module over $\mathfrak{O}$ (the $p$th cyclotomic extension of $\mathbb{Z}_p$) as in the previous section.  It follows immediately that $[A,G]\cong \mathfrak{O}$; thus  $d([A,G])=p-1$, and  $d(A)=p-1$.  Since $A$ has no $p$-torsion and $A/A^G\cong [A,G]$ (induced by multiplication by $g-1$), we should have $d(A^G)=0$, that is, $A^G=0$.  Thus 
$$A\cong [A,G] \cong  \mathfrak{O}\cong \mathfrak{I}_G,$$ (cf. the paragraph just above Theorem \ref{Classification of cyclic modules}),  concluding the proof of (i).

(ii) Let $T=T(A)$ be the $p$-torsion part of $A$, and $p^n$ be its order.   Plainly, $T$ is invariant under the action of $G$, and so it is a $\mathbb{Z}_pG$-submodule of $A$. Observe also that $ \bar{A}=A/T$ has no $p$-torsion, so it is free  over $\mathbb{Z}_p$ and 
$$A\cong T\oplus \bar{A} \quad  (\mbox{over } \mathbb{Z}_p).$$
This shows that $d(A)=d(T)+d( \bar{A})$; in particular, $d( \bar{A})<p$.  Thus  $\bar{A}\cong  \mathfrak{O}$ or 
 $\bar{A}\cong \mathbb{Z}_p$ by virtue of (i).
 
Let $L/M$ be a presentation of $A$, and  let $M'\subseteq L$ be the submodule  corresponding to $T$; so $M'/M \cong T$ and $L/M'\cong \bar{A}$.  If $\bar{A}\cong  \mathfrak{O}$, then $d(L/M')=p-1$ and
$NL\subseteq M'$.  It follows that $d(M')=1$, and so $M'\cong \mathbb{Z}_p$. Note that $L^G=\{\alpha N\mid \alpha\in \mathbb{Z}_p \}$, and so $NL=L^G\cong \mathbb{Z}_p$. Thus the quotient $M'/NL$ should be finite, but since $L/NL\cong \mathfrak{I}_G$ has no $p$-torsion, we see that $M'=NL=L^G$.  This shows that $T=L^G/M$, and so $M$ is the unique submodule of $L^G\cong \mathbb{Z}_p$ of order $p^n$, that is, $M=p^nL^G$. 
Similarly, in the case $\bar{A}\cong  \mathbb{Z}_p$, we have $d(L/M')=1$ and
$\mathfrak{I}_G\subseteq M'$.  It follows that $d(M')=p-1$.  As $d(\mathfrak{I}_G)=p-1$,
the quotient $M'/\mathfrak{I}_G$ should be finite, but since $L/\mathfrak{I}_G\cong \mathbb{Z}_p$ has no $p$-torsion, we see that $M'=\mathfrak{I}_G$.  This shows that $T=\mathfrak{I}_G/M$, and so $M$ is the unique submodule of $\mathfrak{I}_G\cong \mathfrak{O}$ of order $p^n$, that is, $M=(g-1)^n \mathfrak{I}_G$, as claimed.
\end{proof}


\end{document}